\documentclass[10pt]{amsart}
\usepackage{graphicx}
\usepackage{latexsym}
\usepackage{fancyhdr}
\usepackage{amsmath, amssymb}
 \usepackage[utf8]{inputenc}
\usepackage[all]{xy}
\usepackage{pdflscape}
\usepackage{longtable}
\usepackage{rotating}
\usepackage{hyperref}
\usepackage{subfigure}
\usepackage{tensor}
\usepackage{enumerate}

\usepackage{color}

\theoremstyle{plain}
\newtheorem{thm}{Theorem}[section]

\newtheorem{lem}[thm]{Lemma}

\theoremstyle{definition}
\newtheorem{defi}[thm]{Definition}
\theoremstyle{remark}
\newtheorem{rem}[thm]{Remark}
\numberwithin{equation}{section}

\newcommand{\A}{\mathcal{A}}

\newcommand{\haut}{\operatorname{ht}}
\newcommand{\lgw}{\longrightarrow}
\newcommand{\lgm}{\longmapsto}

\newcommand{\AAA}{\mathbb A}

\newcommand{\ovl}{\overline}
\newcommand{\Id}{\operatorname{Id}}

\newcommand{\ord}{\operatorname{ord}}

\renewcommand{\Im}{\operatorname{Im}}

\newcommand{\wdt}{\widetilde}

\newcommand{\id}{\operatorname{id}}

\renewcommand{\k}{\Bbbk}

\newcommand{\K}{\mathbb{K}}

\newcommand{\I}{\mathcal I}
\newcommand{\N}{\mathbb{N}}

\newcommand{\Q}{\mathbb{Q}}

\newcommand{\h}{\text{H}}
\newcommand{\Deg}{\text{Deg}}

\newcommand{\lag}{\langle}
\newcommand{\rag}{\rangle}

\renewcommand{\phi}{\varphi}

\renewcommand{\P}{\mathbb{P}}

\begin{document}
%
\title{Finiteness results concerning algebraic power series}
\author{Fuensanta Aroca}
\email{fuen@im.unam.mx}
\address{Instituto de Matem\'aticas, Universidad Nacional Aut\'onoma de M\'exico (UNAM), Mexico}

\author{Julie Decaup}
\email{julie.decaup@im.unam.mx}
\address{Instituto de Matem\'aticas, Universidad Nacional Aut\'onoma de M\'exico (UNAM), Mexico}

\author{Guillaume Rond}
\email{guillaume.rond@univ-amu.fr}
\address{Universit\'e publique, France}


\thanks{The third author is deeply grateful to the UMI LASOL of  CNRS where this project has been partially carried out.}

\begin{abstract}
We construct an explicit filtration of the ring of algebraic power series by finite dimensional constructible sets, measuring the complexity of these series. As an application, we give a bound on the dimension of the set of algebraic power series  of bounded complexity
lying on an algebraic variety defined over the field of power series.
\end{abstract}

\maketitle
The ring of polynomials over a field $\k$ is a $\k$-vector space filtered by finite dimensional vector spaces, namely the $\k$-vector spaces of polynomials of degree less than $d$, for every $d\in\N$. In many situations, the description of the ring of polynomials as a ring filtered by finite dimensional vector spaces is very useful, both for theoretical and applied aspects. The main advantage of this, is that the set of polynomials of degree bounded by $d$ is fully described by a finite number of data. One common use of this fact is when one tries to approximate objects by polynomials. For example, this is the case in analysis, by the use of Weierstrass approximation theorem, or in algebra when one deals with approximations by polynomials of solutions of differential or functional equations, or transcendental estimations of such solutions.\\
\\
In several cases, dealing with polynomials is not practical or effective enough, in particular because the implicit function theorem does not hold in the polynomial setting. To avoid this problem, one can replace the ring of polynomials by the ring of algebraic power series over $\k$. This is the set of formal power series that are algebraic over the ring of polynomials. This set is  a ring (satisfying the implicit function theorem), in particular it is a $\k$-vector space, but there is no explicit or effective description of a filtration of this ring by finite dimensional  spaces.\\
\\
In this note, we give an explicit description of a filtration of the ring of algebraic power series by constructible sets. A constructible set is a subset of an affine space $\k^n$ which is described by polynomial equalities and inequalities. Here, the sets of this filtration are the sets $\A(d,h)$ of algebraic power series whose minimal polynomial $P(x,T)$ over $\k[x]$, is such that
$$\deg_x(P)\leq h,\ \deg_T(P)\leq d.$$
We prove that these sets can be seen as  constructible subsets of $\k^{N(d,h)}$ for some constant $N(d,h)$ depending only on $d$ and $h$ (see Corollary \ref{cor}), and the dimension of this constructible set is computed. Roughly speaking, the idea is to try to identify an algebraic power series with its minimal polynomial. But this cannot work directly, since two distinct algebraic power series may have the same minimal polynomial. Moreover, an irreducible polynomial in $x$ and $T$ has no  power series root in general. Therefore there is no correspondence between algebraic power series and irreducible polynomials. To overcome this problem, we first describe  the subset of $\A(d,h)$ of algebraic power series whose minimal polynomial satisfies the implicit function theorem as a constructible subset of some $\k^N$ (see Theorem \ref{Idh constructible}). And from this we deduce the same kind of result for $\A(d,h)$. Let us mention that this kind of approach has already been used in \cite{F} and \cite{HM}.\\
\\
As an application, we give a bound on the dimension of the set of solutions of polynomial equations with coefficients in $\k(\!(x)\!)$ whose entries are in $\A(d,h)$, when $\k$ is algebraically closed (see Theorem \ref{application}). This problem is a non-archimedean analogue of the problem of bounding the number of $\ovl\Q$-points of a complex algebraic variety.\\

\noindent \textbf{Acknowledgment.} 
We are very grateful to Michel Hickel and Micka\"el Matusinski for their comments and suggestions on a previous version of our paper. In particular they pointed out that there was an essential gap in the proof of Theorem \ref{cor}.

\section{Preliminaries}
In the whole paper $\k$ will always denote a field, and $x$ will denote a single indeterminate. The ring of algebraic power series will be denoted by $\k\lag x\rag$. For every integer $n$, $\mathbb A_\k^n$ will denote the affine space of dimension $n$ over $\k$.

\begin{defi}
Let $f\in\k\lag x\rag$. The morphism $\psi:\k[x,T]\lgw \k\lag x\rag$ defined by  $\psi(P(x,T))=P(x,f)$ is not injective and its kernel is a height one prime ideal of $\k[x,T]$. Therefore it is generated by one polynomial. If $P(x,T)$ is  such a generator, any other generator of this ideal is a multiple of $P(x,T)$ by a non-zero element of $\k$. Any such a generator is called a \emph{minimal polynomial} of $f$.  But, by abuse of language, we will always refer to such an element by \emph{the} minimal polynomial of $f$.
\end{defi}

\begin{defi}\label{height}\cite{AB}\cite{Ro}
Let $P(T)\in\k[x,T]$.  The maximum of the degrees of the coefficients of $P(T)$ seen as a polynomial in $T$ is called the \emph{height} of $P$.\\
For $f\in\k\lag x\rag$,  the height of its minimal polynomial is called the \emph{height} of $f$, and is denoted by $\h(f)$. The \emph{degree} of $f$ is the degree of the field extension $\k(x)\lgw \k(x,f)$ or, equivalently, the degree of its minimal polynomial seen as a polynomial in $T$. It is denoted by $\Deg(f)$.
\end{defi}

\begin{defi} We define the following sets:
\begin{itemize}
\item $\A(d,h)$ denotes the set of algebraic power series of degree $\leq d$ and height $\leq h$.
\item $\A(d,h)_0$ denotes the set of algebraic series of $\A(d,h)$ that vanish at 0.
\item $\I(d,h)$ denotes the set of algebraic power series of $\A(d,h)_0$   whose minimal polynomial satisfies the implicit function theorem.
\end{itemize}
That is, $f\in\I(d,h)$, if and only if its minimal polynomial $P(x,T)$ satisfies
$$P(0,0)=0\text{ and }\frac{\partial P}{\partial T}(0,0)\neq 0.$$
In particular we have $\I(d,h)\subset \A(d,h)_0\subset \A(d,h)$ for every $d$, $h$.
\end{defi}

\begin{rem}\label{glob}
It is straightforward to check that $\A(d,h)=\mathbb A_\k\times\A(d,h)_0$.
\end{rem}

\begin{rem}\label{min_poly}
Let $f$ be an algebraic power series with $f(0)=0$ and assume that there is $P\in\k[x,T]$ such that $P(0,0)=0$, $\frac{\partial P}{\partial T}(0,0)\neq 0$ and $P(x,f)=0$. Then the minimal polynomial of $f$ satisfies the implicit function theorem. Indeed such a minimal polynomial is denoted by $Q$ and should divide $P$: $P=QR$ for some polynomial $R$. Then
$$ \frac{\partial P}{\partial T}= \frac{\partial Q}{\partial T}R+ Q\frac{\partial R}{\partial T}.$$
Since $f(0)=0$ and $Q(x,f(x))=0$ then $Q(0,0)=0$. Hence $\left(\frac{\partial Q}{\partial T}R\right)(0,0)\neq 0$ and $\frac{\partial Q}{\partial T}(0,0)\neq 0$.
\end{rem}

\section{Filtration of the ring of algebraic series by constructible sets}
\begin{defi}
Let $\k$ be a field and let $d$, $h$, $e\in\N$. We define
$\A(d,h,e)$ to be the subset of algebraic power series $f\in\A(d,h)_0$ such that if $P$ denotes the minimal polynomial of $f$,  we have
$$\ord\left(\frac{\partial P}{\partial T}(x,f)\right)= e.$$
\end{defi}
So we have $\I(d,h)=\A(d,h,0)$ and  $\A(d,h)_0=\cup_{e\in\N}\A(d,h,e)$. Indeed in positive characteristic the minimal polynomial of an algebraic power series is separable since $\k\langle x\rangle$ is a separable extension of $\k[x]$. This comes from the fact that $\k\langle x\rangle$ is the henselization of the local ring $\k[x]_{(x)}$. Moreover this union is finite by the following lemma:

\begin{lem}\label{bound_der}\cite[Lemma 3]{HM}
If $e> 2dh$ then $\A(d,h,e)=\emptyset$.
\end{lem}


\begin{lem}\label{truncation}
We have an injective map
$$\phi_{d,h,e}:\A(d,h,e)\lgw  \mathbb A_\k^{e}\times\I(d,h+e(d-2))$$
defined as follows:
let $f\in\A(d,h,e)$ and let us write 
$$f=f^{(0)}+x^{e}f^{(1)}$$
where $f^{(0)}$ is a polynomial in $x$ of degree $\leq e$ and vanishing at 0, and $f^{(1)}$ is an algebraic power series vanishing at 0. Then $f^{(1)}\in \I(d,h+e(d-2))$ and $f^{(0)}\in\mathbb A_\k^{e}$ by identifying the set of polynomials of degree $\leq e$ vanishing at 0 with $\mathbb A_\k^{e}$. Therefore we define 
$$\phi_{d,h,e}(f):=(f^{(0)}, f^{(1)}).$$
\end{lem}

\begin{proof}
let $f\in\A(d,h,e)$ and let
$$P=a_dT^d+\cdots+a_1T+a_0$$
be its minimal polynomial. We can write
$$f=f^{(0)}+x^{e}f^{(1)}$$
where $f^{(0)}$ is a polynomial in $x$ of degree $\leq e$ and $f^{(1)}$ is an algebraic power series vanishing at 0.\\
We have that
$$P(f^{(0)}+Z)=P(f^{(0)})+\frac{\partial P}{\partial T}(f^{(0)})Z+Q(x,Z)$$
where $Q(x,Z)$ is a polynomial in $x$ and $Z$, divisible by $Z^2$. Moreover if we write
$$P(f^{(0)}+Z)=b_dZ^d+\cdots+b_0$$
we have that
$$\deg(b_i)\leq h+e(d-i).$$
Let us set 
$$R(T):=\frac{P(f^{(0)}+x^{e}T)}{x^{2e}}=\frac{P(f^{(0)})}{x^{2e}}+\frac{1}{x^{e}}\frac{\partial P}{\partial T}(f^{(0)})T+\frac{Q(x,x^{e}T)}{x^{2e}}.$$
We have $\frac{Q(x,x^{e}T)}{x^{2e}}\in\k[x,T]$ since $Q(x,T)$ is divisible by $T^2$. Moreover since $f^{(0)}-f\in (x^{e+1})$ we have $\frac{\partial P}{\partial T}(f^{(0)})-\frac{\partial P}{\partial T}(f)\in (x^{e+1})$, and because $\frac{\partial P}{\partial T}(f)$ is of order $e$ then $\frac{\partial P}{\partial T}(f^{(0)})$ has order $e$. Finally $R(f^{(1)})=0$ so $\frac{P(f^{(0)})}{x^{2e}}\in\k[x]$. This proves that $R(T)\in\k[x,T]$.\\
We have that
$$R(T)=c_dT^d+\cdots+c_0$$
with
$$\deg(c_i)=\deg(b_i)+ie-2e\leq h+e(d-2)$$
and $\frac{\partial R}{\partial T}(0,0)=c_1(0)\neq 0$. So $f^{(1)}$ is the only power series solution of $R=0$ vanishing at 0 by the implicit function theorem and
$$f^{(1)}\in \I(d,h+e(d-2)).$$
This construction gives an injective map
$$\A(d,h,e)\lgw \mathbb A_\k^{e}\times \I(d,h+e(d-2))$$
as claimed.
\end{proof}
This proves that 
$$\k\langle x\rangle =\bigcup_{e,d,h}\mathbb A_\k^{e}\times\I(d,h)$$
where $\mathbb A_\k^{e}$ is identified with polynomials in $x$ of degree $\leq e$ vanishing at 0 and the inclusion $\mathbb A_\k^{e}\times\I(d,h)\subset\k\langle x\rangle$ is given by
$$(f,g)\lgm f+x^eg.$$

\begin{lem}
We have  $\mathbb A_\k^{e}+\I(d,h)\subset \A(d,h+ed)$ for all integers $d$, $h$, $e$.
\end{lem}

\begin{proof}
Let $e\in \N$ and let us consider a polynomial $R(Z)\in\k[x,Z]$, of degree $d'\leq d$, such that
$\frac{\partial R}{\partial Z}(0,0)\neq 0$. Assume that
$\deg_xR\leq h$.  By the implicit function theorem, $R=0$ has a unique algebraic power series solution denoted by $f^{(1)}$. Now let $f^{(0)}\in\k[x]$ be a polynomial of degree $\leq e$ vanishing at 0. We set
$$P(T)=x^{ed'}R\left(\frac{T-f^{(0)}}{x^e}\right).$$
Then $P(f^{(0)}+x^ef^{(1)})=0$. Moreover if
$$R(T)=c_{d'}T^{d'}+\cdots+c_0$$
then 
$$P(T)=c_{d'}(T-f^{(0)})^{d'}+c_{d'-1}x^e(T-f^{(0)})^{d'-1}+\cdots+c_0x^{d'e}=a_{d'}T^{d'}+\cdots+a_0$$
where 
$$\deg(a_i)\leq \max_{j\geq i}\left\{\deg(c_{j})+e(d'-j)+\deg(f^{(0)})(j-i)\right\}\leq h+e(d'-i).$$
So $f:=f^{(0)}+x^ef^{(1)}\in \A(d',h+ed')\subset \A(d,h+ed)$.
\end{proof}


\begin{lem}\label{univ}
Let $\k$ be a field. 
Let $f\in \I(d,h)$ and let $P=\sum_{i\leq h, j\leq d}a_{i,j}x^i T^j$ be a polynomial satisfying the implicit function theorem and vanishing at $f$. Let us consider 
$$P_{d,h}:=\sum_{i\leq h, j\leq d}A_{i,j}x^i T^j$$ where the $A_{i,j}$ are new indeterminates and $A_{0,0}$ is assumed to be 0. Then $P_{d,h}$ has a unique power series solution
$$f_{d,h}\in (x)\k\left\langle \frac{A_{i,j}}{A_{0,1}}, x\right\rangle$$
where $(i,j)$ runs over $\{0,\ldots, d\}\times\{0,\ldots, h\}\backslash \{(0,0), (0,1)\}$. Moreover, the coefficients of the $x^k$ in the expansion of $f_{d,h}(x)$ are in $\k\left[\frac{A_{i,j}}{A_{0,1}}\right]$, and we have 
$$f_{d,h}\left(\frac{a_{i,j}}{a_{0,1}},x\right)=f(x).$$
\end{lem}

\begin{proof}
The existence and the unicity of $f_{d,h}$ comes from the implicit function theorem.\\
First we prove that $f_{d,h}\in\k\left[\frac{A_{i,j}}{A_{0,1}}\right][[ x]]$. Indeed we have that 
$$P(d,h)\left(\frac{A_{i,j}}{A_{0,1}}, x,0\right)\in (x)\text{ and }\frac{\partial P}{\partial T}\left(\frac{A_{i,j}}{A_{0,1}}, x,0\right)\notin (x).$$
 So by Hensel Lemma the root $f_{d,h}$ belongs to the completion of $\k\left[\frac{A_{i,j}}{A_{0,1}},x\right]$ with respect to the ideal $(x)$, i.e. $f_{d,h}\in\k\left[\frac{A_{i,j}}{A_{0,1}}\right][[ x]]$. This implies that 
$f_{d,h}\left(\frac{a_{i,j}}{a_{0,1}},x\right)$ is well defined.\\
Finally we conclude that $f_{d,h}\left(\frac{a_{i,j}}{a_{0,1}},x\right)=f(x)$ since $f_{d,h}$ is the unique solution of $P_{d,h}=0$ vanishing at $x=0$ and $f$ is the unique solution of $P=0$ vanishing at $x=0$.
\end{proof}


Since there are $(d+1)(h+1)-2=dh+d+h-1$ indeterminates $\frac{A_{i,j}}{A_{0,1}}$
the proposition defines a surjective map 
 $$f_{d,h}:\mathbb A_\k^{dh+d+h-1}\lgw \I(d,h).$$
 This map is not injective since different polynomials can have the same power series solution.\\
 In fact let us identify the set of polynomials of $\k[x,T]$ (up to multiplication with a nonzero constant of $\k$) of degree in $x$ less than $h$ and of degree in $T$ less than $d$ with $\mathbb P_\k^{(d+1)(h+1)-1}$ with homogeneous coordinates $A_{i,j}$. Then  $\mathbb A_\k^{dh+d+h-1}$ is identified with the set of polynomials $P(x,T)$ such that $P(0,0)=0$ and $\frac{\partial P}{\partial T}(0,0)\neq 0$. Thus here $\mathbb A_\k^{dh+d+h-1}$ is the intersection of the affine open chart $A_{0,1}\neq 0$ with the  hypersurface $A_{0,0}=0$.\\
 \\
 
We can state our first main result concerning the structure of $\I(d,h)$:

\begin{thm}\label{Idh constructible} Assume that $\k$ is algebraically closed. We have the following properties:
\begin{itemize}
\item[1)] For every $d\geq 1$ and every $h\geq 1$, there is an injective map
$$\psi_{d,h}:\I(d,h)\lgw  \mathbb A_\k^{dh+d+h-1}$$
whose image is a constructible subset $\mathcal C_{d,h}$  that contains a non empty open subset of $\mathbb A_\k^{dh+d+h-1}$. \\
Here we identify $\mathbb A_\k^{dh+d+h-1}$ with the set of polynomials $P(x,t)$ such that
$$P(0,0)=0, \frac{\partial P}{\partial t}(0,0)=1, \deg_x(P)\leq h, \deg_t(P)\leq d,$$
and the map $\psi_{d,h}$ is defined by identifying $\I(d,h)$ with the subset of irreducible polynomials  in $\mathbb A_\k^{dh+d+h-1}$.\\
Moreover, we have
$$f_{d,h}\circ\psi_{d,h}=\id_{\I(d,h)}.$$
\item[2)] 
For every $d'\geq d$, $h'\geq h$, we denote by $\pi_{d',d,h',h}^{(1)}:\I(d',h')\lgw \I(d,h)$ and by $\pi_{d',d,h',h}^{(2)}: \mathbb A_\k^{d'h'+d'+h'-1}\lgw \mathbb A_\k^{dh+d+h-1}$ the canonical projection maps. Then we have that
$$\psi_{d,h}\circ\pi_{d',d,h',h}^{(1)}=\pi_{d',d,h',h}^{(2)}\circ\psi_{d',h'}.$$

\end{itemize}

\end{thm} 

\begin{rem}
We can see in the proof that the map $\psi_{d,h}$ is the map sending a series of $\I(d,h)$ onto the vector of coefficients of its minimal polynomial, which is normalized in the sense that  the coefficient of $x^0T^1$ is $1$.
\end{rem}

\begin{proof}
The map $f_{d,h}:\mathbb A_\k^{dh+d+h-1}\lgw \I(d,h)$ is not injective since different polynomials may have the same root. To get an injective map, we need to restrict the map to the set of irreducible polynomials. Indeed, by Remark \ref{min_poly}, the minimal polynomial of $f\in\I(d,h)$ satisfies the implicit function theorem.  Therefore we have to remove from $\mathbb A_\k^{dh+d+h-1}$ the points corresponding to the reducible polynomials. We can do it as follows:\\
For every integers $d_1$, $d_2$, $h_1$, $h_2$ with $d_1+d_2\leq d$ and $h_1+h_2\leq h$ set 
$$Q_{d_1,h_1}=\sum_{i\leq h_1, j\leq d_1}q_{i,j}x^i T^j,\ R_{d_2,h_2}=\sum_{i\leq h_2, j\leq d_2}r_{i,j}x^i T^j,$$
for some variables $q_{i,j}$ and $p_{i,j}$. Then the product $Q_{d_1,h_1}R_{d_2,h_2}$ is a polynomial $P=\sum_{i\leq h, j\leq d}a_{i,j}x^i T^j$ where the $a_{i,j}$ are polynomials in the $q_{i,j}$ and $r_{i,j}$. The product defines a rational map
$$\Phi_{d_1,d_2,h_1,h_2}:\mathbb A_\k^{(d_1+1)(h_1+1)}\times \mathbb A_\k^{(d_2+1)(h_2+1)}\lgw \mathbb A_\k^{(d+1)(h+1)}$$
whose image can be identified with the polynomials $P$, with $\deg_T(P)\leq h_1+h_2$ and $\deg_x(P)\leq d_1+d_2$, that are the product of 2 polynomials of degrees less than $(h_1,d_1)$ and $(h_2,d_2)$. In fact it is straightforward to check that this map is defined by bi-homogeneous polynomials  so it induces a rational map
$$ \mathbb P_\k\Phi_{d_1,d_2,h_1,h_2}:\mathbb P_\k^{(d_1+1)(h_1+1)-1}\times \mathbb P_\k^{(d_2+1)(h_2+1)-1}\lgw \mathbb P_\k^{(d+1)(h+1)-1}$$
Here we want to consider only polynomials $P$ whose constant term is zero. If such a polynomial is the product of two polynomials $Q$ and $R$ then the constant term of $Q$ or of $R$ has to be zero. Thus here we set
$$C_{d_1,d_2,h_1,h_2}=\Im({\mathbb P_\k\Phi_{d_1,d_2,h_1,h_2}}_{|\mathbb P_\k^{(d_1+1)(h_1+1)-2}\times \mathbb P_\k^{(d_2+1)(h_2+1)-1}})\cup $$
$$\cup\Im({\mathbb P_\k\Phi_{d_1,d_2,h_1,h_2}}_{|\mathbb P_\k^{(d_1+1)(h_1+1)-1}\times \mathbb P_\k^{(d_2+1)(h_2+1)-2}})$$
and the intersection of this set with $\mathbb A_\k^{dh+d+h-1}$ corresponds exactly to the set of polynomials $P$ with $P(0,0)=0$ and $\frac{\partial P}{\partial T}(0,0)\neq0$ that are the product of two polynomials whose degrees are coordinatewise less than or equal to  $(d_1,h_1)$ and $(d_2,h_2)$.\\
Let us remark that the dimension of $\mathbb P_\k^{(d_1+1)(h_1+1)-2}\times \mathbb P_\k^{(d_2+1)(h_2+1)}$ is 
$$(d_1+1)(h_1+1)+(d_2+1)(h_2+1)-3.$$
But we have that
\begin{equation}\label{ineg}\begin{split}(d_1+d_2)(h_1+h_2)+d_1+d_2&+h_1+h_2-1-\left((d_1+1)(h_1+1)+(d_2+1)(h_2+1)-3\right)\\
&=d_1h_2+d_2h_1\geq 0\end{split}\end{equation}
So if $d_1+d_2<d$ or $h_1+h_2<h$ then  $\dim(C_{d_1,d_2,h_1,h_2})<(d+1)(h+1)$.\\
 If $d_1+d_2=d$ and $h_1+h_2=h$  we have equality in \eqref{ineg} if and only if
$$(d_1,h_1) \text{ or } (d_2,h_2)=0,$$
$$(d_1,d_2)=0, $$
$$\text{ or } (h_1,h_2)=0.$$
In the first case $Q$ or $R$ is a nonzero constant of $\k$. In the second case we have that $P=QR$ is of degree $0$ in $T$ and so it does not correspond to a point in $\mathbb A_\k^{dh+d+h-1}$. In the last case $P=QR$ is of degree $0$ in $x$ so  its roots are in the algebraic closure of $\k$, and the only algebraic power series vanishing at 0 which is in $\k$ is 0, and this case cannot occur if $h\geq 1$.\\
Hence, 
in all the cases we need to consider, we have 
\begin{equation}\label{dim}\dim(C_{d_1,d_2,h_1,h_2})<dh+d+h-1.\end{equation} 
So, when $h\geq 1$, we can identify $\I(d,h)$ with
$$\mathcal C_{d,h}:=\mathbb A_\k^{dh+d+h-1}\backslash \bigcup _{\begin{array}{c}\scriptscriptstyle d_1,d_2,h_1,h_2\\ 
\scriptscriptstyle d_1+d_2\leq d, h_1+h_2\leq h\\\scriptscriptstyle d_1h_2+d_2h_1>0\end{array}} C_{d_1,d_2,h_1,h_2}$$
which is a constructible set by Chevalley's Theorem. Finally this former set contains an open subset of $\mathbb A_\k^{dh+d+h-1}$ by \eqref{dim}. This proves 1).\\
\\
Let $d'\geq d$ and $h'\geq h$. 
From the construction of $f_{d',h'}$ and $\psi_{d',h'}$, it is straightforward to see that 
$$f_{d',h'_{|\AAA_\k^{dh+d+h-1}}}=f_{d,h}\ \text{ and }\  \psi_{d',h'_{|\AAA_\k^{dh+d+h-1}}}=\psi_{d,h}.$$
This proves 2).
\end{proof}

Our second result describes the structure of $\A(d,h)_0$ as a constructible set.
\begin{thm}\label{cor}
Assume that $\k$ is algebraically closed. Let $d$, $h$, and $e\in\N$. Then we have the following properties:
\begin{enumerate}
\item[i)] The image $\mathcal  C_{d,h,e}$  of the injective map
$$\left(\Id_{\mathbb A_\k^e}\times\psi_{d,h+e(d-2)}\right)\circ \phi_{d,h,e}:\A(d,h,e)\lgw \mathbb A_k^e\times\mathcal  C_{d,h+e(d-2)}$$ 
is constructible in $\mathbb A_k^e\times \mathcal C_{d,h+e(d-2)}$. 
\item[ii)] The set $\A(d,h)_0$ can be identified with the constructible set
$$\bigcup_{e\leq 2dh}\mathcal  C_{d,h,e}\subset \mathbb A_\k^{N(d,h)}$$
for some $N(d,h)\in\N$.
\item[iii)]
The dimension of this constructible set is $dh+d+h-1$.
\end{enumerate}
\end{thm}

\begin{rem}
In particular, the map 
$$\left(\left(\Id_{\mathbb A_\k^e}\times\psi_{d,h+e(d-2)}\right)\circ\phi_{d,h,e}\right)^{-1}:\mathcal  C_{d,h,e}\lgw \A(d,h,e)$$
is a regular map in the sense that all the coefficients of the series $f\in\A(d,h,e)$ are polynomial functions into the coordinates on $\mathbb A_\k^{N(d,h)}$ (by Lemma \ref{univ}) and are well defined on $\mathcal  C_{d,h,e}$.
\end{rem}

\begin{rem}\label{emb_dim}
In the proof we show that we can choose 
$$N(d,h)=2dh(d-2)(d+1)+3dh+d+h-1.$$
\end{rem}

\begin{proof}[Proof of Theorem \ref{cor}]
We do the following: let $f^{(0)}\in \k^e$, $f^{(1)}\in \I(d,h+e(d-2))$, and $f:=f^{(0)}+x^ef^{(1)}$. We  want to find necessary and sufficient conditions for $f$ to be in $\A(d,h,e)$, that is we want to describe the image of the map $\phi_{d,h,e}$ defined in Lemma \ref{truncation}.\\
Let $R(Z)\in\k[x,Z]$ be a polynomial such that
$$\deg_Z(R)\leq d,\ \deg_x(R)\leq h+e(d-2),$$
$$R(f^{(1)})=0,\text{ and } \frac{\partial R}{\partial Z}(0,0)\neq 0.$$
We assume that $R(Z)$ is irreducible. For such a $f\in\A(d,h,e)$,
 we set
$$P(T)=x^{ed}R\left(\frac{T-f^{(0)}}{x^e}\right).$$
Then $P(f^{(0)}+x^ef^{(1)})=0$, i.e $P$ vanishes at $f$. Moreover $\deg_x(P)\leq h+ed$. 
The map
$$(f^{(0)},R(T))\lgm  P(T):=x^{ed}R\left(\frac{T-f^{(0)}}{x^e}\right)$$
is a polynomial map from $\AAA_\k^e\times\P_\k^{(d+1)(h+e(d-2)+1)-1}$ to $\P_\k^{(h+ed+1)(d+1)-1}$. \\
\\
Let $S(T)$ be an irreducible polynomial vanishing at $f$. Since the following field extensions degrees are equal:
$$[\k(\!(x)\!)(f):\k(\!(x)\!)]=[\k(\!(x)\!)(f^{(1)}):\k(\!(x)\!)],$$
we have $\deg_T(S)=\deg_T(R)$. Since $\deg_T(P)=\deg_T(R)$, there is a polynomial $a(x)\in\k[x]$ such that
\begin{equation}\label{cond1}P(x,T)=a(x)S(x,T),\end{equation}
\begin{equation}\label{cond2}\frac{\partial P}{\partial T}(x,0)=a(x)\frac{\partial S}{\partial T}(x,0).\end{equation}
Then, $f\in\A(d,h,e)$ if and only if $S$ and $R$ are irreducible, $\deg_x(S)\leq h$ and $\ord\left(\frac{\partial S}{\partial T}(x,f^{(0)})\right)=e$.\\
\\
We denote by $\P_\k^{d_1}$ (resp. $\P_\k^{(d_2+1)(d+1)-1}$) the projective space of nonzero polynomials of $\k[x]$ (resp. $\k[x,T]$) of degree $\leq d_1$ (resp. $\leq d_2$ in $x$ and $\leq d$ in $T$) modulo multiplication by elements of $\k^*$. For every $d_1$, $d_2$, $d$, $e$, $h$, we denote by $V_{d_1,d_2,d,e,h}$ the algebraic set of elements
 $$(a(x),S(x,T),f^{(0)},R(T))\in \P_\k^{d_1}\times\P_\k^{(d_2+1)(d+1)-1}\times \AAA_\k^e\times \P_\k^{(d+1)(h+e(d-2)+1)-1}$$
  such that
\begin{equation}\label{key_eq}a(x)S(x,T)-x^{ed}R\left(\frac{T-f^{(0)}}{x^e}\right)=0.\end{equation}
Recall that $f\in\A(d,h,e)$ is identified with $(f^{(0)}, R(x,T))\in \AAA_\k^e\times \P_\k^{(d+1)(h+e(d-2)+1)-1}$. Therefore $f\in \A(d,h,e)$ if and only if there is $(a(x),S(x,T))\in \P_\k^{d_1}\times\P_\k^{(d_2+1)(d+1)-1}$, for some $d_1$, $d_2$ with $d_1+d_2\leq h+ed$ and $d_2\leq h$,  such that 
$$\left\{\begin{matrix} R(x,T) \text{ and } S(x,T) \text{ are irreducible,}\\
R(0,0)=0,\ \frac{\partial R}{\partial T}(0,0)\neq 0,\\
(a(x),S(x,T),f^{(0)}(x),R(x,T))\in V_{d_1,d_2,d,e,h}\\
\text{ and } \ord\left(\frac{\partial S}{\partial T}(x,f^{(0)})\right)=e.\end{matrix}\right.$$
The set of irreducible polynomials  $S(x,T)$, with $\deg_x(S)\leq d_2$, $\deg_T(S)\leq d$, is a constructible subset  $C^1_{d_2, d}\subset \P_\k^{(d_2+1)(d+1)-1}$  as shown in the proof of Lemma \ref{Idh constructible}. Moreover, by Lemma \ref{Idh constructible}, the set of irreducible polynomials $R(x,T)$ such that $R(0,0)=0$, $\frac{\partial R}{\partial T}(0,0)\neq 0$, $\deg_x(R)\leq h$, $\deg_T(R)\leq d$, is a constructible subset $C^2_{d,h}$ of $\P_\k^{(d+1)(h+e(d-2)+1)-1}$.\\
The condition  $\ord\left(\frac{\partial S}{\partial T}(x,f^{(0)})\right)=e$ defines a constructible set 
$$C'_{d_2,d,e}\subset \P_\k^{(d_2+1)(d+1)-1}\times \AAA_\k^e.$$
Now, we consider the following projections:
$$\pi_{3,4}:  \P_\k^{d_1}\times\P_\k^{(d_2+1)(d+1)-1}\times \AAA_\k^e\times \P_\k^{(d+1)(h+e(d-2)+1)-1}\lgw  \AAA_\k^e\times \P_\k^{(d+1)(h+e(d-2)+1)-1}$$
$$\pi_{2,3}: \P_\k^{d_1}\times\P_\k^{(d_2+1)(d+1)-1}\times \AAA_\k^e\times \P_\k^{(d+1)(h+e(d-2)+1)-1}\lgw \P_\k^{(d_2+1)(d+1)-1}\times\AAA_\k^e$$
and we set
$$E_{d_1,d_2,d,e,h}:=\left( \P_\k^{d_1}\times C'_{d_2,d,e}\times \P_\k^{(d+1)(h+e(d-2)+1)-1}\right)\cap V_{d_1,d_2,d,e,h}\cap\left( \P_\k^{d_1}\times C^1_{d_2,d}\times  \AAA_\k^e\times C^2_{d,h}\right).$$
Then we have that $\A(d,h,e)$ is equal to
\begin{equation*} \bigcup_{\begin{matrix}\scriptscriptstyle d_1+d_2\leq h+ed\\ \scriptscriptstyle d_2\leq h\end{matrix}}\pi_{3,4}\left(E_{d_1,d_2,d,e,h}\right)\end{equation*}
that is a constructible subset $\mathcal C_{d,h,e}$ of $\mathbb A_k^e\times\mathbb A_k^{(d+1)(h+e(d-2)+1)-1}$.\\
\\
We claim that ${\pi_{3,4}}_{|_{E_{d_1,d_2,d,e,h}}}$ is injective. Indeed, for $(f^{(0)},R(x,T))\in \A(d,h,e)$, there is a unique couple $(a(x),S(x,T))$ such that 
$(a(x),S(x,T), f^{(0)}, R(x,T))\in E_{d_1,d_2,d,e,h}$, by    \eqref{key_eq} and because $S(x,T)$ has to be irreducible.
 \\
On the other hand, the fiber of $\pi_{2,3}$ over $(S(x,T),f^{(0)})\in C'_{d_2,d,e}\cap \left(C^1_{d_2,d}\times\AAA_\k^e\right)$ is  finite. Indeed, if there is $(a(x),R(x,T))\in  \P_\k^{d_1}\times \P_\k^{(d+1)(h+e(d-2)+1)-1}$ such that 
$$a(x)S(x,T)-x^{ed}R\left(\frac{T-f^{(0)}}{x^e}\right)=0,$$
then, $f^{(0)}+x^ef^{(1)}$ is a root of $S(x,T)$ where $f^{(1)}$ is the unique solution of $R(x,T)=0$ vanishing at $x=0$. Therefore $\pi_{2,3}^{-1}((S,f^{(0)}))$ is finite because $S(x,T)$ has a finite number of roots. \\
Therefore the dimension of $E_{d_1,d_2,d,e,h}$
is equal to the dimension of its image under $\pi_{2,3}$. 
First let us assume that $(S(x,T),f^{(0)})\in \pi_{2,3}(E_{d_1,d_2,d,e,h})$. Indeed, by \eqref{key_eq}, we have that $S(x,f^{(0)}+x^e f^{(1)})=0$ where $f^{(1)}$ is the unique solution of some polynomial equation $R(x,T)=0$ with $R(0,0)=0$ and $\frac{\partial R}{\partial T}(0,0)\neq 0$. Therefore $S(0,0)=0$.
We denote by $C^3_{d_2,d}$ the set of polynomials $S$ such that $S(0,0)=0$. Then $\pi_{2,3}(E_{d_1,d_2,d,e,h})$ is included in $C'_{d_2,d,e}\cap (C^3_{d_2,d}\times \AAA_\k^e)$.\\
We are going to bound the dimension of $C'_{d_2,d,e}\cap (C^3_{d_2,d}\times \AAA_\k^e)$ as follows: \\
We denote by $F_k$ the coefficient of $x^k$ in the expansion of $f^{(0)}$, for $1\leq k\leq e$. We denote by $S_{i,j}$ the coefficient of $x^iT^j$ in $S(x,T)$, for $0\leq i\leq d_2$ and $0\leq j\leq d$.  
The coefficient of $x^l$ in $S(x,f^{(0)})$ is a polynomial $G_l$ with integer coefficients depending on the indeterminates $F_k$ for $k\leq l$, and $S_{i,j}$ for $i\leq l-1$. Therefore $C'_{d_2,d,e}\cap (C^3_{d_2,d}\times \AAA_\k^e)$ is defined by the equations :
$$\left\{\begin{matrix} S_{0,0}=0\\
G_l(F_k, S_{i,j})_{ k\leq l, i\leq l}=0,\ \ l<e\\
G_e(F_k, S_{i,j})_{k\leq e, i\leq e-1}\neq 0 \end{matrix}\right.$$
and the ideal defining the Zariski closure of $C'_{d_2,d,e}\cap (C^3_{d_2,d}\times \AAA_\k^e)$ is the radical ideal of the ideal generated by $S_{0,0}$ and the $G_l$ for $l<e$. For any $p<e$, we set
$$I_p:=\langle S_{0,0}, G_0,\ldots, G_p\rangle$$
$$\text{and } I_{-1}=\langle S_{0,0}\rangle.$$
These are ideals depending only on the following indeterminates (if $p\geq 0$):
$$F_k, \text{ for }k\leq p, \text{ and } S_{i,j} \text{ for } i\leq p.$$
Moreover $G_p$ has the form
$$G_p=S_{p,0}+\wdt G_p(F_k, S_{i,j})_{k\leq p, i\leq p-1}.$$
Therefore we have, for all $p\leq e$:
$$\haut(I_{p-1})\leq \haut(I_p\cap\Q[F_k,S_{i,j}]_{k\leq p-1, i\leq p-1})<\haut(I_p).$$
In particular we have that $\haut(I_e)\geq e+2$. Hence
$$\dim(C'_{d_2,d,e}\cap (C^3_{d_2,d}\times \AAA_\k^e))\leq \left[e+(d+1)(d_2+1)-1\right]-e-2\leq (d+1)(h+1)-3.$$
Therefore 
$$\dim(\A(d,h,e))< (d+1)(h+1)-2.$$
Moreover, when $e=0$, we know that $\mathcal{A}(d,h,0)=\mathcal{I}(d,h)$ and by Theorem \ref{Idh constructible}, this set is a constructible set of dimension $(d+1)(h+1)-2$. Thus,
$$\A(d,h)_0=\bigcup_{e\leq 2dh}\mathcal \A(d,h,e)$$
is a constructible set of dimension equal to $(d+1)(h+1)-2$. 
\end{proof}

\section{Points of bounded complexity in varieties over $\k(\!(x)\!)$}
Let $E$ be a subset of $\k(\!(x)\!)^n$.
For all non negative integers $d$ and $h$, we set
$$E_{d,h}:=E\cap \A(d,h)^n$$
and we denote by $n_{d,h}(E)$ the dimension of the Zariski closure of $E_{d,h}$ in the affine space $\mathbb  A_{\k}^{nN(d,h)}$ where $N(d,h)$ is given in Remark \ref{emb_dim}.

\begin{thm}\label{application}
Let $\k$ be an algebraically closed field and let $X$ be an algebraic subvariety of $\mathbb{A}_{\k(\!(x)\!)}^n$ of dimension $m$. Then
$$n_{d,h}(X)\leq m(dh+d+h).$$
\end{thm}

\begin{rem}
This result is an analogue of Lemma \cite[5.1.1]{CCL} (that gives the same kind of bound on the dimension of the polynomial solutions of degree $\leq d$). The proof is based on the use of linear projections. But, on the contrary of the polynomials of degree $\leq d$, the difficulty comes from the fact that the sets $\A(d,h)$ are not stable by linear change of coordinates with coefficients in $\k$ (since these sets are not $\k$-vector spaces).
\end{rem}

\begin{proof}[Proof of Theorem \ref{application}]
We will prove the statement by induction on $m$. If $m=0$, then $X$ is finite and $n(d,h)=0$.\\
Now assume that the statement is true for every integer less than $m$ and let $X$ be of dimension $m$. For every $E\subset \{1,\ldots, n\}$ with $\operatorname{Card}(E)=m$, we denote by $\pi : \mathbb A^n\lgw \mathbb A^m$
 the projection defined by $\pi_E(x_1,\ldots, x_n):=(x_i)_{i\in E}$. For such a set $E$, we define $B_E$ to be  the subset of points of $\pi_E(X)$ where $\pi_{E|X}$ is not finite. Then $\pi_E(X)\setminus B_E$ contains an open set of the Zariski closure of $\pi_E(X)$ by the upper semi-continuity of the dimension of the fibers of $\pi_E$. \\
Therefore the dimension of the Zariski closure of the set of $\A(d,h)$-points of $X\setminus \pi_E^{-1}(B_E)$ has dimension less or equal to $\dim(\A(d,h)^m)$, which is equal to $m(dh+d+h)$ by Remark \ref{glob} and  Theorem \ref{cor}. 
Thus we can replace $X$ by $X\cap \pi_E^{-1}(B_E)$. We repeat this operation for every $E$ as above and we assume that
$$X\subset \bigcap_{E\subset \{1,\ldots, n\},\ \operatorname{Card}(E)=m}\pi_E^{-1}(B_E).$$
By considering the different irreducible components of the $B_E$ separately, we may assume that all the $B_E$ are irreducible. Therefore the ideal of $\pi_E^{-1}(B_E)$ is a prime ideal $J_E$ generated by polynomials depending only on the $x_i$ for $i\in E$.\\
\\
We fix such a set $E$ that we denote by $E_1$. We also fix an index $i_1\in E_1$ such that $J_{E_1}$ is generated by polynomials of $\K[x_i,i\in E_1]$ but not by polynomials of $\K[x_i, i\in E_1\setminus\{i_1\}]$.\\
Then we pick $E_2\subset \{1,\ldots,n\}\setminus\{i_1\}$ such that $\operatorname{Card}(E_2)=m$, and we fix an index $i_2\in E_2$ such that $J_{E_2}$ is generated by polynomials of $\K[x_i,i\in E_2]$ but not by polynomials of $\K[x_i, i\in E_2\setminus\{i_2\}]$.\\
We repeat this process and construct a sequence 
$$E_1,\ldots, E_{n-m+1}\subset \{1,\ldots, n\}$$
of subsets of cardinal $m$, and a sequence 
$$i_1\in E_1\setminus (\cup_{i=2}^{n-m+1} E_i),\  i_2\in E_2\setminus (\cup_{i=3}^{n-m+1} E_i),\ldots, i_{n-m+1}\in E_{n-m+1}$$
such that $J_{E_k}$ is generated by polynomials of $\K[x_i,i\in E_k]$ but not by polynomials of $\K[x_i, i\in E_k\setminus\{i_k\}]$. Hence, by Krull's principal ideal theorem we have
$$\haut\left(J_{E_1}+J_{E_2}+\ldots+J_{E_{n-m+1}}\right)\geq n-m+1.$$
Equivalently we have
$$\dim(\cap_{i=1}^{n-m+1}\pi_{E_i}^{-1}(B_{E_i}))\leq m-1.$$
Therefore, the result follows by induction.
\end{proof}


\end{document}